\documentclass[a4paper,12pt]{amsart}
\usepackage{graphicx}

\newtheorem{theorem}{Theorem}

\newtheorem{proposition}{Proposition}

\newtheorem*{thmA}{Theorem}

\def\beq{\begin{equation}}
\def\eeq{\end{equation}}
\numberwithin{equation}{section}

\newcommand{\kap}{\mathrm{cap}}

\newcommand{\R}{{\mathbb R}}
\newcommand{\C}{{\mathbb C}}

\newcommand{\D}{{\mathbb D}}

\newcommand{\U}{{\mathcal U}}

\title[The univalent Bloch-Landau constant]{The univalent Bloch-Landau constant, harmonic symmetry and conformal glueing}
\author{Tom Carroll}
\address{Department of Mathematics, National University of Ireland, Cork, Ireland.} 
\email{t.carroll@ucc.ie}
\author{Joaquim Ortega-Cerd\`a}
\address{Departament de Matem\`atica Aplicada i An\`alisi, Universitat de Bar\-ce\-lo\-na, Gran Via 585, 08007 Barcelona, Spain.}
\email{jortega@ub.edu}
\keywords{Univalent Bloch-Landau constant; Conformal welding; Harmonic symmetry}

\begin{document}
\maketitle
\begin{abstract}
By modifying a domain first suggested by Ruth Goodman in 1935 and by exploiting
the explicit solution by Fedorov of the Poly\'a-Chebotarev problem in the case
of four symmetrically placed points, an improved upper bound for the univalent
Bloch-Landau constant is obtained.  The domain that leads to this improved bound
takes the form of a disk from which some arcs are removed in such a way that the
resulting simply connected domain is harmonically symmetric in each arc with
respect to the origin. The existence of domains of this type is
established, using techniques from conformal welding, and some general
properties of harmonically symmetric arcs in this setting are established.
\end{abstract}

\section{The univalent Bloch-Landau constant and harmonic symmetry.}
 We write $R_D$ for the supremum radius of all disks contained in a planar
domain $D$, this geometric quantity is called the {\sl inradius\/} of the
domain. We write $\D$ for the disk with centre zero and radius one in the complex
plane. 
 
Let us suppose that $f$ is a univalent map of the unit disk
$\D$. There is a number $\U$, independent of $f$, such that
\beq\label{3}
R_{f(\D)} \geq \U \vert f'(0)\vert.
\eeq
Thus the image of the unit disk under any univalent map $f$ of $\D$ contains
some disk of every radius less than $\U\vert f'(0)\vert$. The number $\U$, known
as the univalent or schlicht Bloch-Landau constant, is the largest number for
which \eqref{3} holds, in that if $U>\U$ then there is a conformal mapping $f$
of the unit disk for which $f(\D)$ contains no disk of radius $U\vert
f'(0)\vert$.  This constant was introduced in 1929 by Landau \cite{Landau},
following on from Bloch's famous paper \cite{Bloch} of a few years earlier. It
is a consequence of the Koebe one-quarter theorem that $\U \geq 1/4$. Landau himself
proved $\U > 0.566$ in \cite{Landau}. Over time, Laudau's estimate for $\U$ was
improved by Reich \cite{Reich} ($\U > 0.569$), Jenkins \cite{Jenkins1}
($\U>0.5705$),  Toppila \cite{Top} ($\U> 0.5708$), Zhang \cite{Zhang} and Jenkins
\cite{Jenkins2} ($\U>0.57088$). Most recently, Xiong \cite{Xiong} has proved that
$\U > 0.570884$. Over the years, several domains have been put forward that provide
upper bounds for $\U$, among them those of Robinson \cite{ROB1} ($\U < 0.658$)
in 1935, Goodman \cite{Good} ($\U < 0.65647$) in 1945 and, most recently, Beller
and Hummel \cite{BH} ($\U < 0.6564155$) in 1985. Our first  result is an improved upper bound for $\U$.  
\begin{theorem}\label{t1}There is a simply connected domain $D_0$ that has inradius 1, and a conformal map $f$ of the unit disk $\mathbb{D}$ onto $D_0$ for which
$$
\U \leq \frac{1}{\vert f^{'}(0)\vert} \leq 0.6563937.
$$
\end{theorem}
The significance of this result is not so much the numerical improvement in the
upper bound for $\U$, but rather the shape of the domain that produced it, which is shown in Figure~\ref{real}. 
\begin{figure}[ht]
\begin{center}
\includegraphics{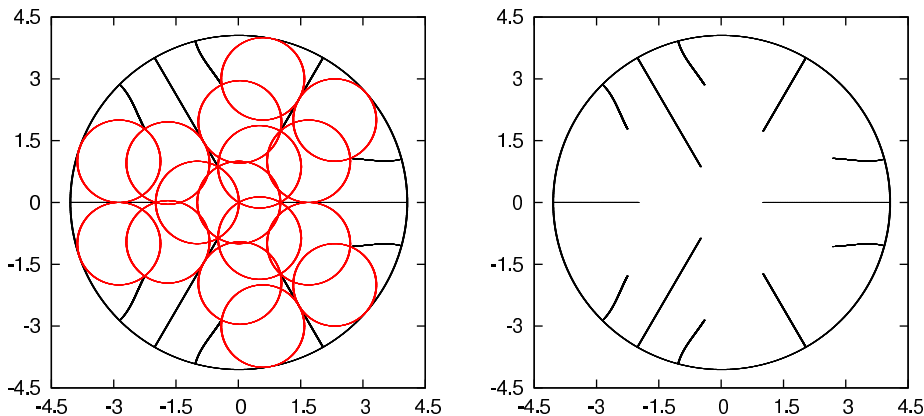}
\caption{The domain $D_0$ and the extremal disks}\label{real}
\end{center}
\end{figure}

We may write 
\beq\label{4}
\U = \inf\left\{ \frac{R_{f(\D)}}{\vert f^{'}(0)\vert} \colon f \mbox{ is
univalent in }\D\mbox{ and } f(0) = 0\right\}.
\eeq
This infimum is attained. If $f$ is univalent in $\D$ with $f(0)=0$, $f^{'}(0)=1$, and if
$R_{f(\D)} = \U$, then $f$ is a {\sl Bloch function of the third kind\/} and the
domain $D = f(\D)$ is said to be an {\sl extremal domain\/} for the inequality \eqref{3}.  A proof that extremal domains exist was first written down explicitly by Robinson \cite{ROB2}. Jenkins \cite{Jenkins2} has proved  that an extremal domain must contain an extremal disk, that is, a disk of
radius $\U$.  In \cite{Jenkins3}, Jenkins described a condition that any extremal domain for the univalent Bloch-Landau inequality \eqref{3} must satisfy. This condition was extended by the first author in \cite{C}. In order to describe this more general condition, we need the notion of harmonic symmetry. 

A simple $\mathcal{C}^1$ arc $\gamma$ is said to be an {\sl internal boundary arc} for a domain $D$ if $\gamma$ is part of the boundary of $D$ and if, to each non-endpoint $\zeta$ of $\gamma$, there corresponds a positive $\epsilon$ such that the disk with centre
$\zeta$ and radius $\epsilon$ is part of $D\cup \gamma$. In the case that $D$ is
simply connected, each non-endpoint $\zeta$ of the arc $\gamma$ corresponds to
two prime ends of $D$ and each has a Poisson kernel associated with it, which we
denote by $P_1(\zeta, z)$ and $P_2(\zeta, z)$.  We say that $D$ is
\emph{harmonically symmetric in $\gamma$ with respect to $z_0$\/} if 
\[
P_1(\zeta, z_0) = P_2(\zeta, z_0) \mbox{ whenever } \zeta \mbox{ is a
non-endpoint of }\gamma.
\]
For example, if $D = \D\setminus\gamma$ then $D$ is harmonically
symmetric in the arc $\gamma$ with respect to $0$ if and only if $\gamma = [r,1]$ for some $r$ in $(0,1)$, up to rotation. More generally, if $\gamma$ is
an internal boundary arc for $D$ and if also $D\cup\gamma$ is simply connected then
$D$ is harmonically symmetric in $\gamma$ with respect to $z_0$ if and only if
$\gamma$ is in a geodesic arc through $z_0$ in the hyperbolic metric for
$D\cup\gamma$.  This can be seen by conformally mapping  $D\cup \gamma$ onto
$\D$ so that $z_0$ corresponds to $0$ and $\gamma$ corresponds to an arc $\tilde\gamma$, and then using the conformal invariance of the Possion kernel to conclude that $\D\setminus\tilde\gamma$ is harmonically symmetric in $\tilde\gamma$ with respect to $0$.

The extension of Jenkins' condition in \cite{C} shows that there is a close relationship between the univalent Bloch-Landau constant and harmonic symmetry. 
\begin{thmA}
Suppose that $D$ is an extremal domain  for the univalent Bloch-Landau constant.
Suppose that $\gamma$ is an internal boundary arc for $D$, no point of which
lies on the boundary of an extremal disk. Then $D$ is harmonically symmetric in
$\gamma$ with respect to $0$.
\end{thmA}
The domains that were constructed in \cite{ROB1, Good, BH} in order to obtain upper
bounds for $\U$ are essentially disks with radial slits removed. The above extremality  condition
suggests how the domains in \cite{ROB1, Good, BH} might be modified so as to make them closer to being extremal, and in turn leads to Theorem~\ref{t1}. 

Harmonic symmetry arises in connection with problems other than the
determination of extremal domains for the univalent Bloch-Landau inequality. It
previously appeared in the work of Betsakos \cite[Proposition~2.1]{Bet} in
relation to another extremal problem, that of maximizing $\vert f'(0)\vert$ over
the family of all conformal maps $f$ of the unit disk, with $f(0)=0$, onto
simply connected subdomains of the unit disk whose complement must contain some
specified points. But the idea goes back much further than this, to Lavrentiev
\cite{Lav} and Gr\"otzsch \cite{Grotzsch}, in the context of the
P\'olya-Chebotarev problem \cite{Polya} that consists in determining the continuum
that has minimal capacity and that contains a given finite set of points in $\C$.
In Lavrentiev's formulation of harmonic symmetry, the preimage of each subarc of an internal boundary arc of the simply connected domain $D$, under a conformal map of the disk onto $D$ under which $0$
corresponds to $z_0$ in $D$, will comprise of two arcs of equal length on the
unit circle. Thus the two \lq sides\rq\ of each subarc of the internal boundary
arc have the same harmonic measure at $z_0$. This is also the formulation
adopted in \cite{C}. 

In our second main result we study domains formed when the disk is
slit along simple arcs in such a way that the resulting domain is simply
connected and is harmonically symmetric in each arc with respect to $0$. We show that the harmonic measure of each arc may be specified, together with the harmonic measure between  the endpoints of the arcs on the unit circle.  To be precise, we consider families $\Gamma$ consisting of a finite number of simple arcs that do not intersect, do not pass through the origin, and lie inside the unit disk $\D$ except for one endpoint of each arc that lies instead on the unit circle.  For the purposes of this paper, we call such a family of arcs \lq admissible\rq. We associate with $\Gamma$ the domain $D(\Gamma)$ that is the complement of the traces of the arcs in the family, so that $D(\Gamma)$ is a simply connected domain containing 0. Our second result concerns the problems of existence and uniqueness in this context. Together with a conformal mapping of the unit disk, it can be used to introduce harmonically symmetric slits in more general simply connected domains. 

\begin{theorem}\label{t2}
Suppose that $n$ positive numbers $a_1$, $a_2$, $\ldots$, $a_n$ and $n$ non-negative numbers $b_1$, $b_2$, $\ldots$, $b_n$ are specified with
\[
\sum_{k=1}^n a_k < 1 \quad\mbox{and}\quad \sum_{k=1}^n a_k + \sum_{k=1}^n b_k = 1.
\]
There is an admissible family of real analytic arcs $\Gamma = \{\gamma_1, \gamma_2, \ldots, \gamma_n\}$  such that 
\begin{itemize}
\item[(T2.1)] each arc $\gamma_k$ has harmonic measure $a_k$ at $0$ with respect to $D(\Gamma)$,
\item[(T2.2)] the domain $D(\Gamma)$ is harmonically symmetric in each arc $\gamma_k$ with respect to $0$,
\item[(T2.3)] the endpoints of the arcs on the unit circle, which we denote by $\zeta_1$, $\zeta_2$, $\ldots$, $\zeta_n$, respectively, are in anticlockwise order and, moreover, the harmonic measure at $0$ and with respect to $D(\Gamma)$ of the anticlockwise arc of the unit circle from $\zeta_k$ to $\zeta_{k+1}$ is $b_k$ for $k = 1$, $2$, $\ldots$, $n-1$. 
\end{itemize}
This configuration is unique up to rotation.
\end{theorem}

The plan of the paper is as follows. We briefly describe Goodman's domain and its modification by Beller and Hummel in the next section. In Section~\ref{sec3}, we construct the domain $D_0$ and prove Theorem~\ref{t1}. In order to do so, we use an explicit solution by Fedorov of the P\'olya-Chebotarev problem in the case of four symmetrically placed points. Theorem~\ref{t2} is proved in  Section~\ref{sec4}, using techniques drawn from conformal welding.  Related results on harmonic symmetry are also established in this section. 
%
%
\section{Goodman's domain and the Beller-Hummel domain}\label{sec2}
Ruth Goodman's domain \cite{Good} is constructed in stages. The first stage
consists of the removal from the plane of three radial halflines that start from
the cube roots of unity. The second stage consists of the removal of three
further radial halflines starting from two times the cube roots of $-1$. The
domain $G_2$ formed by the plane minus these six halflines is shown in
Figure~\ref{fig1}. Goodman continues the construction by, at each stage,
removing radial halflines that bisect the sectors formed by previous generations of halflines in such a way as to maintain inradius~ $1$. 

The circle $C_1$ with unit radius and with centre $P_1 = (c,1)$, where
$c=1+\sqrt{2\sqrt{3}-3}$, is tangent to the halfline $[1,\infty]$ and passes
through the tip $2e^{i\pi/3}$ of the halfline above it. Thus the boundary of Goodman's domain includes a halfline with argument $\pi/6$ and one endpoint on $C_1$, together with the successive rotations of this halfline through an angle $\pi/3$. 
\begin{figure}[ht]
\begin{center}
\includegraphics{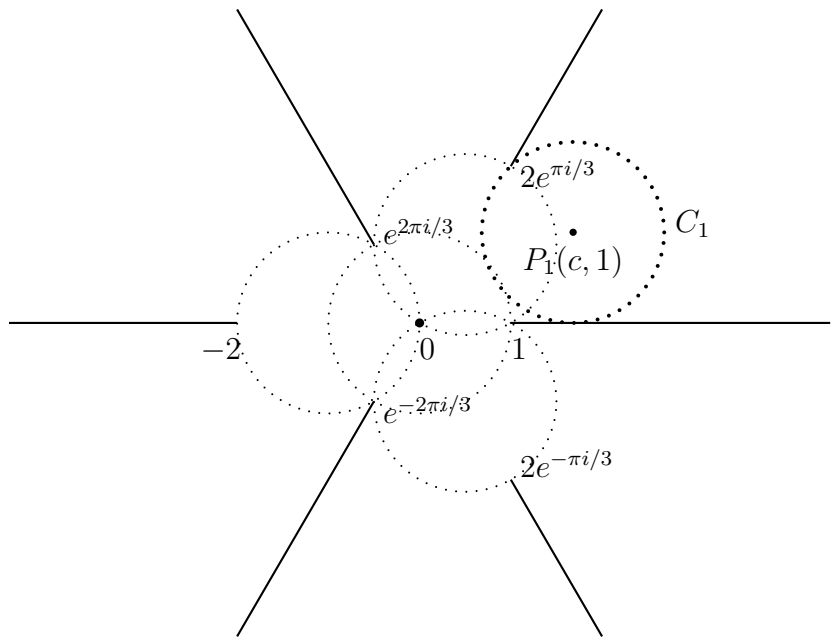}    
\caption{The first two stages of Goodman's domain}\label{fig1}
\end{center}
\end{figure}
The modification of the Goodman domain constructed by Beller and Hummel
\cite{BH} to obtain their upper bound for $\U$ agrees with Goodman's domain up to the
second generation of halflines -- indeed, it is difficult to imagine (but
apparently equally difficult to prove) that the construction of an extremal
domain might begin any differently. Their improved estimate was motivated by the
observation that the circle $C_1$ in Figure~\ref{fig1}  sneaks slightly around
the end of the halfline with angle $\pi/3$ so that its centre does not have argument
$\pi/6$. This led them to vary the angles of the third generation of halfline slits in
Goodman's domain to find an optimal configuration of this type. In their
configuration, the third generation of slits are far from being bisectors of the
six sectors in the domain $G_2$. 
%
%
\section{An improved upper bound for the univalent Bloch-Landau
constant}\label{sec3}
All authors, who either put forward a putative extremal domain for the
Bloch-Landau constant or who seek a numerical upper bound for $\U$, start from
the six-slit plane $G_2$, as described in Section~\ref{sec2}, and proceed by
inserting further radial halflines to divide the sectors as they widen with the
aim of preserving the inradius. It is necessary to truncate at some point when
seeking an upper bound, which we do. It is now clear from \cite{C} (see also the
concluding remark in \cite{BC}) that any new boundary arcs need to be inserted in such a
way that the final domain is harmonically symmetric in each arc. It is not clear
that this can be achieved in an iterative manner, in that the insertion of later
boundary arcs may destroy the harmonic symmetry of earlier arcs. Nevertheless, at least
from a computational point of view, it is natural to begin with the domain $G_2$
and to insert six extra arcs to obtain a domain that is harmonically symmetric
in each new arc with respect to $0$ and  remains symmetric under reflection, and
therefore harmonically symmetric, in each of the original six halflines that
form the boundary of the domain $G_2$. The domain we construct is of the type
shown in Figure~\ref{FigU}. To perform the necessary calculations, we exploit 
the connection between harmonic symmetry and the P\'olya-Chebotarev problem, in particular results of Fedorov \cite{Fed}.

\subsection{Fedorov's results on certain configurations of minimal capacity}
Given $\alpha$ and $c$ such that $0<\alpha\le \pi/2$ and $0<c<2\cos \alpha$,
Fedorov finds the continuum $E(\alpha,c)$  with minimal capacity that contains
each of the points $0$, $c$, $e^{i\alpha}$ and $e^{-i\alpha}$. The typical
extremal configuration is shown in Figure~\ref{fed1fig}.
\begin{figure}[ht]
\begin{center}
\includegraphics{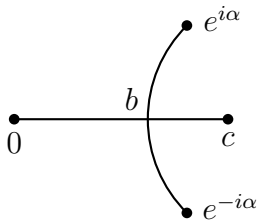}    
\caption{Fedorov's continuum of minimal capacity containing four specified
points: $0$, $e^{i\alpha}$, $e^{-i\alpha}$ and $c$.}
\label{fed1fig}
\end{center}
\end{figure}
The point $b$ is determined explicitly by Fedorov in terms of $c$ and $\alpha$.
Moreover the capacity of the extremal compact set is
\begin{equation}\label{4.1}
\kap\big(E(\alpha,c)\big)= \frac {(1+p)^2 \Theta^2(0)}{4\, c\, \Theta^2(w)}.
\end{equation}
Here $\Theta$ is the Jacobi Theta function \cite[p.\ 577]{AbramStegun}, 
\begin{align*}
p&=\sqrt{1-2c\cos\alpha+c^2},\\
w&=F\left(\arccos\left(\frac{1-p}{1+p}\right);k\right),
\end{align*}
where
\[
k=\sqrt{\frac{p+1-c\cos\alpha}{2p}}
\]
and the function $F$ is an incomplete elliptic integral of the first kind, that
is 
\[
F(x;k)=\int_0^x \frac{dt}{(1-t^2)(1-k^2t^2)}.
\]

\subsection{The required conformal mapping}The mapping function $h$ of the
complement of a compact set $E$ onto the complement of the closed unit disk may
be expanded as 
\beq\label{4.2}
h(z) =  \frac{z}{\kap(E)} + O(1), \quad \mbox{ as } z \to \infty,
\eeq
up to rotation. This provides a link between the problem solved by Fedorov and
the example that will yield an improved upper bound for the Bloch-Landau
constant, in that minimising the capacity of the set therefore corresponds to
maximising the derivative of the mapping $h$ at infinity. Moreover, following
the argument in \cite{C}, the arcs making up the extremal configuration will be
harmonically symmetric at infinity. As noted in the introduction, this latter
observation was first made by Lavrentiev \cite{Lav}. 

We work with domains $\Omega = \Omega_{z_0,R}$ as shown in
Figure~\ref{FigOmega},  where $R>3$, $\vert z_0\vert< R^3$, and the arc
$\gamma_{z_0}$ is chosen so that $\Omega_{z_0,R}$ is harmonically symmetric in
$\gamma_{z_0}$ with respect to $0$. 
\begin{figure}[ht]
\begin{center}
\includegraphics{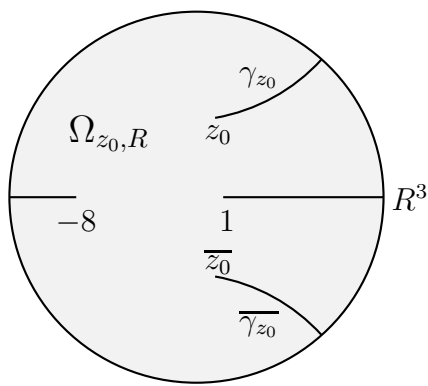}    
\caption{A domain $\Omega_{z_0,R}$}
\label{FigOmega}
\end{center}
\end{figure}
If $g$ is a conformal map of the unit disk $\D$ onto such a domain
$\Omega_{z_0,R}$, with $g(0)=0$, then 
$f(z)=z\sqrt[3]{\frac {g(z^3)}{z^3}}$ is a conformal map of $\D$ onto a domain
$U_{w,R}$ as shown in Figure~\ref{FigU}.
\begin{figure}[ht]
\begin{center}
\includegraphics{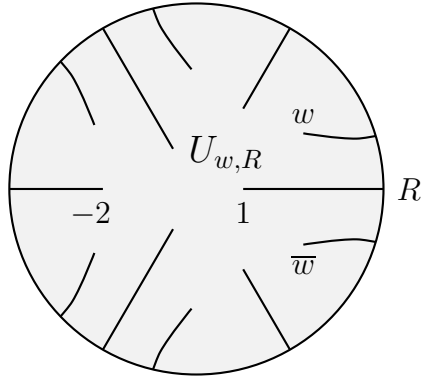}  
\caption{A domain $U_{w,R}$}
\label{FigU}  
\end{center}
\end{figure}
The arcs that appear are all harmonically symmetric, and thus the conformal
mapping of the unit disk onto $U_{w,R}$ is a good candidate for having a
relatively large derivative at the origin. This derivative is
$|f'(0)|=\sqrt[3]{|g'(0)|}$. In order that this provide a useful estimate of the
Bloch-Landau constant, we need to arrange for $U_{w,R}$ to have inradius $1$. We
leave this aside for the moment and show how to use Fedorov's results on
capacity to compute $\vert f'(0)\vert$ for given $z_0$ and $R$. We write $k$ for
the Koebe mapping $k(z) = z/(1-z)^2$  of the unit disk onto the plane slit along
the negative real axis from minus infinity to $-1/4$.

\begin{proposition}\label{Prop2}
We write $f$ for a conformal map of the unit disk $\D$ onto $U_{w,R}$ for
which $f(0) = 0$.  Then, with $z_0 = w^3$,   
\beq\label{4.2a}
\vert f'(0)\vert = \frac{1}{R} \sqrt[3]{\vert \psi(z_0) - \psi(1)\vert\,
\kap\big(E(\alpha,c)\big)}
\eeq
where
\beq\label{4.2b}
\psi(z) = -\frac{1}{k(z/R^3)}
\eeq
and 
\beq\label{4.2c}
e^{i\alpha} = \frac{\psi(z_0) - \psi(1)}{\vert \psi(z_0) - \psi(1)\vert},
\qquad 
c = \frac{\psi(-8) - \psi(1)}{\vert \psi(z_0) - \psi(1)\vert}.
\eeq
\end{proposition}
\begin{proof}
The map 
\beq\label{4.2d}
\phi(z) = \frac{\psi(z) - \psi(1)}{\vert \psi(z_0) - \psi(1)\vert}, 
\eeq
where $\psi$ is given by \eqref{4.2b}, maps $\Omega_{z_0,R}$ onto the complement
of Fedorov's continuum $E(\alpha, c)$  with $\alpha$ and $c$ given by
\eqref{4.2c}. The harmonic symmetry of arcs is preserved because each mapping
extends continuously to all internal boundary arcs and because the domains
involved are all symmetric with respect to the real axis.  If $h$ is the mapping
of the complement of $E(\alpha,c)$ onto the complement of the unit disk
mentioned in \eqref{4.2}, then a suitable mapping $g$ of $\Omega_{z_0,R}$ onto
the unit disk, with $g(0) = 0$,  is given by 
\beq\label{4.3}
g(z) = \frac{1}{h\big( \phi(z)\big)}.
\eeq
Now $g'(0)$ can be computed explicitly, in terms of $z_0$, $R$ and $h'(\infty)$,
for example by computing the power series for $g$. One obtains
\beq\label{4.4}
\vert g'(0) \vert = \frac{\vert \psi(z_0) - \psi(1)\vert}{R^3\,h^{'}(\infty)}.
\eeq
The value of $h'(\infty)$ is $1/\kap\big(E(\alpha,c)\big)$ and is given
explicitly, in its turn, by Fedorov's result \eqref{4.1}. \end{proof}
\subsection{The choice of $w$ and $R$} 
We know how to build, for any $w$ and $R$, a conformal map $f$ of the unit disk
onto 
$U_{w,R}$ and have a formula for its derivative at $0$. In order that this
provide an upper bound for the univalent Bloch-Landau constant, we must choose
$w$ and $R$ in such a way that 
\begin{itemize}
\item The domain $U_{w,R}$ has inradius one,
\item The derivative $|f'(0)|$ is as big as possible.
\end{itemize}

\begin{proposition}\label{Prop3}
We suppose that $R$ is fixed with $3< R < 4.5$. We set 
\beq\label{4.5}
P_2=(R-1)e^{i\left(\frac{\pi}{3}-\alpha\right)}\quad \mbox{where}\quad \alpha =
\arcsin\big[1/(R-1)\big],
\eeq
and write $d = \vert P_2-P_1\vert$, where $P_1$ is as in Section~\ref{sec2}. We
set $\theta = \arccos(d/2)$ and set  
\beq\label{4.6}
w= P_1+ \frac{1}{d}\,e^{-i\theta}\big(P_2-P_1\big).
\eeq
Then $U_{w,R}$ has inradius one. 
\end{proposition}
\begin{figure}[ht]
\begin{center}
\includegraphics{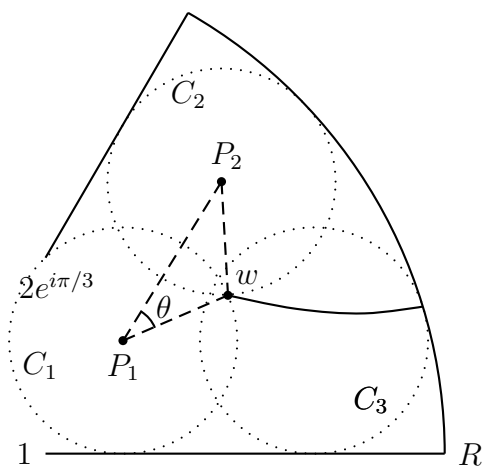}    
\caption{The choice of the point $w$ in Proposition~\ref{Prop3}.}\label{Fig-w}
\end{center}
\end{figure}
\begin{proof}
The circle $C_1$ is the same as that shown in Figure~\ref{fig1}: it has centre
$(c,1)$, where $c=1+\sqrt{2\sqrt{3}-3}$, has unit radius and is tangent to the
slit $[1,R]$ and passes through the tip $2e^{i\pi/3}$ of the slit above it. The
circle $C_2$ with centre $P_2$, as specified in \eqref{4.5}, and of unit radius
is tangent to the circle $\vert z \vert=R$ and tangent to the halfline of
argument $\pi/3$.  It is elementary  to check that if $R\leq4.5$ then $d = \vert
P_2-P_1\vert < 2$, so that these two circles meet, as in Figure~\ref{Fig-w}. The
position of the intersection point $w$ can be computed explicitly, which results in \eqref{4.6}.

Let $C_3$ be the circle of unit radius that is tangent to the circle $\vert z
\vert=R$ and also tangent to the line segment $[1,R]$. Considering that $w$ lies
on its boundary, the domain $U_{w,R}$ will have inradius $1$ if the circle $C_3$
contains the point $w$ (the position of the harmonically symmetric arc through
$w$ being irrelevant to this consideration). In fact the circles $C_2$ and $C_3$
meet at points on the halfline with argument $\pi/6$. Thus $w$ lies inside the
circle $C_3$ if it has argument less than $\pi/6$: it is elementary to check
that this is the case if  $R\leq 4.5$. 
\end{proof}

\begin{proof}[Proof of Theorem~\ref{t1}]
With the choice $R=4.0546358$ in Proposition~\ref{Prop3}, it is then the case
that $U_{w,R}$ has inradius $1$ and Proposition~\ref{Prop2} can be used to
compute the derivative of the conformal map $f$ of the unit disk $\mathbb{D}$
onto $U_{w,R}$ for which $f(0)=0$. This leads to the value $0.65639361315219$,
correct to 10 decimal places, for this derivative, which proves
Theorem~\ref{t1}.

For the actual picture of the domain, shown below in Figure~\ref{real}, 
one has to draw Fedorov's domain and transport it to $U_{w,R}$ with the inverse of the
map $\phi$ given in \eqref{4.2d}. The key point for drawing the domain is to
compute the curvilinear arc of $E(\alpha,c)$. The point $b$ is given explicitly
by Fedorov \cite[Formula 6]{Fed}: 
\[
 b(\alpha,c)=\sqrt{p}\, \frac{\Theta'(w)}{\Theta(w)} + \frac{c}{p+1}.
\]
The quotient $\Theta'(w)/\Theta(w)$ is the Jacobi Zeta function and can be
computed numerically (see Abramovitz and Stegun, \cite[p.\ 578]{AbramStegun}).
To determine the arcs that grow from $b$ to $e^{i\alpha}$ one integrates the
quadratic differential equation $z'(t)^2 Q(z(t))=1$, where 
\[
 Q(z)=\frac{(z-b)^2}{z(z-c)(z^2-2z\cos\alpha +1)},
\]
\cite[Formula 5]{Fed} -- we point out for the reader's convenience that there is
a typograhic error in the referenced formula in that the term $z^2 -2c\cos\alpha
+1$ in the denominator of the quadratic differential should read $z^2
-2z\cos\alpha +1$: Fedorov's Formula~(14) for $Q$ is correct. 
\end{proof}
\section{Harmonic symmetry and conformal glueing}\label{sec4}

\subsection{Proof of Theorem~\ref{t2}} We first prove the existence of a family of arcs with the required properties, and postpone a proof of the uniqueness statement. We begin by assuming that each $b_k$ is strictly positive. 

We divide the unit circle into $2n$ arcs $I_1$, $J_1$, $I_2$, $J_2$, $\ldots$,  $I_n$, $J_n$, in anti-clockwise order so that $\vert I_k\vert  = 2\pi a_k$, $\vert J_k\vert  = 2\pi b_k$, for $1 \leq k \leq n$. 
We write $\phi_k$ for the following involution; $\phi_k : I_k \to I_k$ so that $\zeta$ and $\phi_k(\zeta)$ are at the same distance from the centre of $I_k$, but lie on opposite sides of the centre. Our goal is to produce a conformal map $f$ of $\D$ into $\D$, continuous on the closure of $\D$, that glues each of these involutions, in that 
\beq\label{3.1}
f(\zeta) = f\big(\phi_k(\zeta)\big),\quad \zeta\in I_k, \quad 1 \leq k \leq n.
\eeq
We produce a quasiconformal glueing to begin with, and then correct this to a conformal glueing in a  standard way (see, for example, \cite[Remark~8]{Bishop}). 

We divide a second unit circle into $2n$ equal arcs $\tilde I_1$, $\tilde J_1$, $\tilde I_2$, $\tilde J_2$, $\ldots$,  $\tilde I_n$, $\tilde J_n$, in anti-clockwise order. Next we construct a quasi-symmetric homeomorphism $T$ of the first unit circle to the second unit circle such that $T(I_k) = \tilde I_k$, $T( J_k) = \tilde J_k$ for each $k$ and $T$ is linear on each interval $I_k$ and $J_k$. By the Beurling-Ahlfors Extension Theorem, $T$ may be extended to a quasiconformal map of $\D$ onto $\D$, which we again call $T$, with $T(0) = 0$. 

Next we write $\Gamma_1$ for the admissible family of arcs formed by the straight line segment $\tilde\gamma_1 = [r,1]$, together with is rotations $\tilde\gamma_{k+1} = e^{2\pi k i/n} \tilde\gamma_1$, $k = 1$, $2$, $\ldots$, $n-1$, where $r$ is chosen so that $\omega\big(0,\tilde\gamma_1;D(\Gamma_1)\big) = 1/(2n)$. We write $g$ for the conformal map of $\D$ onto $D(\Gamma_1)$ for which $g(0)=0$ and $g(\tilde I_k) = \tilde\gamma_k$, for $1\leq k \leq n$. Then 
$$
S = g \circ T
$$
is a quasiconformal map of the unit disk onto $D(\Gamma_1)$. The map $S$ extends continuously to the boundary of $\D$. Tracing the boundary correspondence under the mappings $T$, and then $g$, 
shows that it is a quasiconformal glueing of the intervals $I_k$, in that
\[
S(\zeta) = S\big(\phi_k(\zeta)\big),\quad \zeta\in I_k, \quad 1 \leq k \leq n.
\]
By the Measurable Riemann Mapping Theorem,  we can now correct $S$ to a conformal glueing by making a quasiconformal map $R$ of $\D$ onto $\D$ for which $R(0) = 0$ and 
\[
f = R \circ S
\]
is conformal. Then \eqref{3.1} holds. The arcs we are looking for are then $\gamma_k = R(\tilde\gamma_k)$, for $1 \leq k \leq n$ and $\Gamma = \{\gamma_1, \gamma_2, \ldots, \gamma_n\}$. Since $f$ is conformal and $f(0) = 0$, it preserves harmonic measure at $0$. Thus 
$$
\omega\big(0,\gamma_k;D(\Gamma)\big) = \omega\big(0,f^{-1}(\gamma_k);\D\big) =  \omega\big(0,I_k;\D) = a_k. 
$$
This is (T2.1). Moreover, since $f$ is a conformal glueing on the interval $I_k$, the domain $D(\Gamma)$ is harmonically symmetric in $\gamma_k$ with respect to $0$, which is (T2.2). Finally, $f$ maps each arc $J_k$ on the unit circle onto the anti-clockwise arc of the unit circle joining the endpoints $\zeta_k$ of $\gamma_k$ and $\zeta_{k+1}$ of $\gamma_{k+1}$ on the unit circle, so that (T2.3) follows. The fact that the arcs are real analytic comes from the control we have on the quasiconformality of the maps constructed.

This construction depends continuously on the parameters $a_k$ and $b_k$. In the limit as one or more  of the parameters $b_k$ approach zero, it leads to a configuration satisfying (T2.1), (T2.2) and (T2.3) in which two or  more of the curves have a common endpoint on the unit circle. This covers the proof of existence in all cases in Theorem~\ref{t2}.

We now deal with the uniqueness of the configuration we have just now constructed. We suppose that $\Gamma^1 = \{\gamma^1_1, \gamma^1_2, \ldots, \gamma^1_n\}$ and $\Gamma^2 = \{\gamma^2_1, \gamma^2_2, \ldots, \gamma^2_n\}$ are sets of admissible arcs that satisfy (T2.1), (T2.2) and (T2.3) (with $\Gamma$ replaced by $\Gamma^1$ and by $\Gamma^2$ as necessary).  We assume that the arcs in $\Gamma^1$ are real analytic, but those in $\Gamma^2$ need not be. Having set up the intervals $I_1$ to $I_n$ and $J_1$ to $J_n$ as in the proof of existence, we consider conformal mappings $f_1\colon \D \to D(\Gamma^1)$ and $f_2\colon \D \to D(\Gamma^2)$ with $f_1(0) = 0$, $f_2(0)=0$ and with $f_1(I_k) = \gamma_k^1$, $f_2(I_k) = \gamma_k^2$, for each $k$ between $1$ and $n$. Then $f = f_2\circ f_1^{-1}$ is a conformal map from $D(\Gamma^1)$ to $D(\Gamma^2)$. Moreover, $f$ extends continuously from $D(\Gamma^1)$ to $\D$ because of the harmonic symmetry. By the regularity of the arcs in $\Gamma^1$, Morera's Theorem is applicable and we may deduce that $f$ extends to a conformal self map of the disk $\D$ with $f(0)=0$. Hence $f$ is a rotation and the uniqueness statement follows. 

\subsection{A variation on Theorem~\ref{t2}}We describe a version of Theorem~\ref{t2} in which the harmonically symmetric arcs again have specified harmonic measures, but in which one specifies the lengths, rather than the harmonic measures, of the arcs on the unit circle that are formed by the endpoints $\zeta_k$ of the arcs $\gamma_k$. In other words, the position of the endpoints of the arcs $\gamma_k$ on the unit circle may be specified. 
\begin{theorem}\label{t3}
Suppose that $n$ positive numbers $a_1$, $a_2$, $\ldots$, $a_n$, with
$\sum_{k=1}^n a_k < 1$, and $n$ points $\zeta_1$, $\zeta_2$, $\ldots$, $\zeta_n$ in anticlockwise order on the unit circle, are specified. Then there is an admissible family of real analytic arcs $\Gamma = \{\gamma_1, \gamma_2, \ldots, \gamma_n\}$ such that 
\begin{itemize}
\item[(T3.1)]each arc $\gamma_k$ has harmonic measure $a_k$ at $0$ with respect to $D(\Gamma)$,
\item[(T3.2)]the domain $D(\Gamma)$ is harmonically symmetric in each arc $\gamma_k$ with respect to $0$,
\item[(T3.3)] the endpoint of $\gamma_k$ on the unit circle is $\zeta_k$, for each $k$. 
\end{itemize}
\end{theorem}
\begin{proof}
We write $b$ for $1-\sum_1^n a_k$. We consider all possible configurations in Theorem~\ref{t2} in which the numbers $a_1$, $a_2$, $\ldots$, $a_n$ are as specified and the non-negative numbers $b_1$, $b_2$, $\ldots$, $b_n$ are allowed to vary subject to $ \sum_1^n{b_k} = b$. For a permissible choice of the parameters $b_k$, $k=1$, $2$, $\ldots$, $n$, we denote by $x_k$ the endpoint of the resulting arc $\gamma_k$ on the unit circle, and we write $l_k$ for the length of the anticlockwise arc of the unit circle between $x_k$ and $x_{k+1}$ (with $x_{n+1} = x_1$). In this way, we have a map $T$ from the simplex 
\[
\Sigma_1 = \left\{ (b_1, b_2, \ldots , b_n): b_k \geq 0 \mbox{ and } \sum_1^n{b_k} = b\right\}
\]
to the simplex 
\[
\Sigma_2 = \left\{ (l_1, l_2, \ldots , l_n): l_k \geq 0 \mbox{ and } \sum_1^n{l_k} = 2\pi\right\}
\]
given by $T(b_1, b_2, \ldots , b_n) = (l_1, l_2, \ldots , l_n)$. The map $T\colon \Sigma_1 \to \Sigma_2$ is continuous. It has the key property that $l_k$ is zero  if and only if $b_k$ is zero, from this it follows that each vertex, edge, and $i$-face of $\Sigma_1$ is mapped into the corresponding vertex, edge, or $i$-face of $\Sigma_2$. The proof will be complete once it is shown that the map $T$ is onto $\Sigma_2$, for then (T3.3) will hold after a rotation if $l_k$ is chosen to be the arc length between $\zeta_k$ and $\zeta_{k+1}$ on the unit circle (with $\zeta_{n+1} = \zeta_1$). 

Let us first consider any two vertices of the simplex $\Sigma_2$ and the edge $e$ joining them. The pre-images of these vertices under $T$ are vertices of $\Sigma_1$, and the image of the edge joining these vertices lies in the edge $e$. By continuity of the map, $T$ is onto $e$. We can now proceed inductively. We consider an $(i+1)$-face $f_1$ of $\Sigma_1$, the $i$-faces that bound it, and the corresponding $(i+1)$-face $f_2$ of $\Sigma_2$.  Assuming that $T$ maps each of these $i$-faces onto the corresponding $i$-face of $\Sigma_2$, it follows from the key property of $T$ that the image of the $i$-faces bounding $f_1$ has winding number $1\,({\rm mod\,} 2)$ about each interior point of $f_2$. 
Consequently, $T$ is onto $f_2$. 
\end{proof}

\subsection{Harmonically symmetric arcs don't grow}It is natural to ask whether the harmonically symmetric arcs $\gamma_k$ in Theorem~\ref{t3} can be described by means of a differential equation of L\"owner type. As our final result, we show that this is not possible, even in the case of two harmonically symmetric arcs. We note that if the data in Theorem~\ref{t3} is symmetric with respect to $\R$, then, by construction, the resulting harmonically symmetric curves can be taken symmetric with respect to $\R$.

\begin{theorem}\label{t4}
We suppose that $\zeta$ lies on the upper half of the unit circle and that $a_1$ and $a_2$ lie in $(0,1)$ with $a_1< a_2$. Theorem~\ref{t3} is applied twice to construct two pairs of harmonically symmetric curves, $\{\gamma_1, \overline\gamma_1\}$ and $\{\gamma_2, \overline\gamma_2\}$ respectively, the first from the data $a_1$, $a_1$, $\zeta$, $\overline{\zeta}$, and the second from the data $a_2$, $a_2$, $\zeta$, $\overline{\zeta}$. Then $\gamma_1 \not \subseteq \gamma_2$ except in the case when $\zeta = i$. 
\end{theorem}

\begin{proof}
Let us suppose that $\gamma_1 \subseteq \gamma_2$. Since $\D\setminus \{\gamma_1 \cup \bar\gamma_1\}$ is harmonically symmetric in $\bar\gamma_1$ with respect to $0$, $\bar\gamma_1$ is a geodesic arc in $\D\setminus\gamma_1$. Since $\D \setminus \{\gamma_2 \cup \bar\gamma_2\}$ is harmonically symmetric in $\bar\gamma_2$ with respect to $0$, $\bar\gamma_2$, and hence its subarc $\bar\gamma_1$, is a geodesic arc in $\D\setminus\gamma_2$. Thus $\bar\gamma_1$ is a geodesic arc with respect to both the domain $\D\setminus\gamma_1$ and the domain $\D\setminus\gamma_2$. 

We write $\Gamma_1$ for the full geodesic in $\D\setminus\gamma_1$ of which $\bar\gamma_1$ is a part and write $\Gamma_2$ for the full geodesic in $\D\setminus\gamma_2$ of which $\tilde\gamma_1$ is a part. Both $\Gamma_1$ and $\Gamma_2$ pass through the origin. We map $\D\setminus\gamma_1$ onto the unit disk by a conformal map $f_1$ so that $\Gamma_1$ is mapped onto $(-1,1)$ and map  $\D\setminus\gamma_2$ onto the unit disk by a conformal map $f_2$ so that $\Gamma_2$ is mapped onto $(-1,1)$. 

We consider the map $g = f_1\circ f_2^{-1}$, which maps the unit disk into itself conformally. Moreover, $g$ is real-valued on $f_2(\bar\gamma_1)$, which itself is a subinterval of $(-1,1)$.  Thus $g$ is real-valued on the entire interval $(-1,1)$. Since $f_2^{-1}(-1,1) = \Gamma_2$, it then follows that 
\[
f_1\left( \Gamma_2\right) \subseteq (-1,1) = f_1\left(\Gamma_1\right),
\]
so that 
\[
\Gamma_2 \subseteq \Gamma_1.
\]

There are now two possible geometric situations to consider, depending on where the geodesic $\Gamma_2$ might end (both it and $\Gamma_1$ begin at the endpoint of $\bar\gamma_1$ on the unit circle). Suppose that $\Gamma_2$ were to end at a boundary point of $\D\setminus\gamma_1$. In this case, $\Gamma_2$ and $\Gamma_1$ coincide and the final step is to recall that the harmonic measure of the boundary is split evenly in two along a hyperbolic geodesic. Let $E$ be that part of the boundary of the domain $\D\setminus\gamma_2$ that is bordered by the common endpoints of $\Gamma_2$ and $\Gamma_1$ and does not include $\gamma_2\setminus\gamma_1$. At any point $P$ on $\Gamma_2$, for example $0$, the harmonic measure of $E$ at $P$ with respect to $\D\setminus\gamma_2$ is $1/2$ since $\Gamma_2$ is a hyperbolic geodesic for $\D\setminus\gamma_2$. Since $P$ lies on $\Gamma_1$, which is a hyperbolic geodesic for $\D\setminus\gamma_1$, $E$ also has harmonic measure $1/2$ at $P$ with respect to $\D\setminus\gamma_1$. However, $\D\setminus\gamma_1$ contains $\D\setminus\gamma$ strictly, so this is impossible. 
We conclude that the geodesic $\Gamma_2$ must end at a point of $\gamma_2\setminus\gamma_1$. 

Under the map $f_1$, the geodesic $\Gamma_1$ relative to the domain $\D\setminus\gamma_1$ is mapped onto the real axis in the unit disk and the domain $\D\setminus\gamma_2$ is mapped to the domain 
\[
\Omega = \D \setminus \left[f_1 \left(\gamma_2\setminus\gamma_1\right)\right].
\]
The arc $f_1 \left(\gamma_2\setminus\gamma_1\right)$ begins on the unit circle and ends at an interior point of the unit disk. Moreover, $f_1(\Gamma_2)$ is a geodesic in $\Omega$ and is part of the real axis, and so $\Omega$ must be symmetric under reflection in the real axis. This forces $f_1 \left(\gamma_2\setminus\gamma_1\right)$ to be part of the interval $(-1,1)$. 
Pulling this picture back under $f_1^{-1}$, we find that 
\[
\Gamma_1 = \Gamma_2 \cup \left(\gamma_2\setminus\gamma_1\right),
\]
so that the geodesic $\Gamma_1$ of $D_1 = \D\setminus \{\gamma_1 \cup \bar\gamma_1\}$ is an arc that joins the endpoints of $\gamma_1$ and $\bar\gamma_1$ and that passes through $0$. Since $\Gamma_1$ divides the boundary of $D_1$ into two, each of harmonic measure $1/2$ at $0$, and the domain $D_1$ is also harmonically symmetric in both $\gamma$ and $\bar\gamma$ with respect to $0$, it follows that the two arcs of the unit circle determined by $\zeta$ and $\bar\zeta$ have equal harmonic measure $b_1 = (1-a_1)/2$ at $0$ with respect to $D_1$. By the uniqueness statement in Theorem~\ref{t2}, there is only one configuration (up to rotation) that realises the configuration in Theorem~\ref{t2} with $n=2$ and symmetric data $a_1$, $a_1$, $b_1$, $b_1$ and this configuration is the disk with the ends of a diameter removed. Since $D_1$ is symmetric in the real axis, we conclude that $\gamma_1$ lies along the imaginary axis. 
\end{proof}

\end{document}